\documentclass[10pt]{amsart}
\usepackage[utf8x]{inputenc}

\usepackage{scalerel,amssymb}

\usepackage[font={footnotesize,it}]{caption}
\usepackage{ulem}
\usepackage{dsfont}
\usepackage{amsmath}
\usepackage{amssymb}
\usepackage{graphicx}
\usepackage{amssymb}
\usepackage{amsfonts}
\usepackage{soul}
\usepackage{mathpazo}
\usepackage{color}
\usepackage{yfonts}
\usepackage{paralist}
\usepackage{stmaryrd}
\usepackage{amsxtra}
\usepackage{latexsym,mathrsfs}
\usepackage{soul}
\usepackage{mathabx}
\usepackage{float}
\usepackage{epsfig, enumerate}
\usepackage{color}
\usepackage{hyperref}

\usepackage{rotating}
\newtheorem{theorem}{Theorem}[section]
\newtheorem{lemma}[theorem]{Lemma}

\newtheorem{coro}[theorem]{Corollary}

\newtheorem{prop}[theorem]{Proposition}

\numberwithin{theorem}{section}
\numberwithin{equation}{section}

\usepackage{amsthm,regexpatch}

\makeatletter

\def\endcproof{%
  \renewcommand\qedsymbol{\openbox\rlap{\textsubscript{\std@currentclaim}}}%
  \endproof
}
\makeatother

\newtheorem{claim}{Claim}[theorem]

\makeatletter
\xpretocmd{\claim}{\std@patch@thm}{}{}
\newcommand{\std@patch@thm}{%
  \makeatletter
  \regexpatchcmd{\@thm}
    {(\c{refstepcounter}\cB.\cP.2\cE.)}
    {\1\c{std@savelabel}}
    {}{}%
  \makeatother
}
\newcommand{\std@savelabel}{\xdef\std@currentclaim{\@currentlabel}}
\makeatother

\newtheorem{conjecture}[theorem]{Conjecture}

\title{On Graph Odd Edge-Colorings and Odd Edge-Coverings}

\author{Xiao-Chuan Liu}
\address[Liu]{Departamento de Matemática,
 Universidade Federal de Pernambuco,
	Avenida Jornalista Aníbal Fernandes - Cidade Universitária, Recife, Brasil}
\email{xiaochuan.liu@ufpe.br}

\author{Mirko Petru\v{s}evski}
\address[Petru\v{s}evski]{Department of Mathematics and Informatics, Faculty of Mechanical Engineering - Skopje, Ss. Cyril and Methodius University, Macedonia} 
\email{mirko.petrushevski@gmail.com}

\author{Xu Yang}
\address[Yang]{Instituto de Computação, Universidade Federal de Alagoas,
	Av. Lourival Melo Mota, S/N, Maceió, Brasil} 
\email{yang@ic.ufal.br}

\begin{document}
\maketitle{}
\begin{abstract}
An odd $k$-edge-coloring of a graph $G$ is a (not necessarily proper) edge-coloring with at most $k$ colors such that each non-empty color class induces a graph in which every vertex is of odd degree; similarly, if more than one color per edge is allowed, we speak of an odd $k$-edge-covering of $G$.
In this paper, we fully resolve two major 
conjectures on odd edge-colorings and odd edge-coverings of graphs,
proposed by Petru{\v{s}}evski and {\v{S}}krekovski ({\it European Journal of Combinatorics,} 91:103225, 2021). 
The first conjecture states that, apart from two particular exceptions which are respectively odd $5$- and odd-$6$-edge-colorable, for any other loopless and connected 
graph $G$ there exists an edge $e$ such that 
$G\backslash \{e\}$ is odd $3$-edge-colorable. 
The second conjecture 
states that any simple graph $G$  admits an odd $3$-edge-covering in which  
at most one edge receives more than one color. In addition, we strongly confirm the second conjecture by demonstrating that there exists an odd $3$-edge-covering in which at most one edge receives two colors and the rest of the edges receive unique colors.
\end{abstract}

\section{Introduction}
In this paper, we consider only finite undirected loopless graphs, which may contain multiple edges unless specified otherwise as simple graphs. Throughout this paper, we employ the graph notation from~\cite{murty2008graph}.

For a graph $G$, an {\it edge-coloring} of $G$ is an assignment of colors from the set $[k] = \{1,2,\ldots,k\}$ to the edges of $G$, 
 which provides an edge decomposition of $G$ into $k$ color classes.
Similarly, an {\it edge-covering} of $G$ is an assignment of colors from $[k]$ to the edges of $G$, allowing for multiple colors per edge. 
Equivalently, we can denote an edge-coloring by a mapping $\varphi: E(G) \to [k]$. In the case of an edge-covering, we denote it by a set-valued mapping $\psi: E(G)\to 2^{[k]}\backslash \emptyset$.
If $v\in V(G)$, $E_G(v)$ is the set of edges incident with $v$. Each $v$ with an even (resp. odd) degree $d_{G}(v)$ is an \textit{even} (resp. \textit{odd}) \textit{vertex} of $G$. A graph is \textit{even} (resp. \textit{odd})
if all its vertices are even (resp. odd). 
We say $G$ is {\it odd $k$-edge-colorable} 
(respectively, {\it odd $k$-edge-coverable}) if there is an edge-coloring (respectively, an edge-covering)
of $G$ by at most $k$ odd subgraphs, which we call an odd $k$-edge-coloring (respectively, an odd $k$-edge-covering).
In other words, the coloring map or the covering map 
$\varphi$ is required to satisfy that
any non-empty edge set 
$E_j := \{ e \in E(G) : e\in \varphi^{-1}(j) \}$, $j=1,\ldots,k$, forms an odd graph. We also define $\chi_{\text{odd}}'(G)$ (respectively, $\text{cov}_{\text{odd}}(G)$) to be the least integer $k_0$ such that $G$ is odd
$k_0$-edge-colorable (respectively, odd $k_0$-edge-coverable). 
Equivalently, $\chi_{\text{odd}}'(G)$ (respectively, $\text{cov}_{\text{odd}}(G)$) is the least integer
$k_0$ such that the edge set of $G$ admits a decomposition (respectively, cover) comprised of $k_0$ odd subgraphs. 
An obvious necessary and sufficient condition for odd edge-colorability or odd edge-coverability of $G$ is the absence of vertices incident only to loops. Apart from this, the presence of loops
does not influence the existence nor changes the values of the parameters $\chi_{\text{odd}}'(G)$ and $\text{cov}_{\text{odd}}(G)$. Therefore, while studying odd edge-colorability and odd edge-coverability, we may without loss of generality restrict ourselves to loopless connected  graphs.
 
The topics of determining $\chi_{\text{odd}}'(G)$ and $\text{cov}_{\text{odd}}(G)$ for a graph $G$ 
have been very active in recent years, and we refer to the recent survey paper~\cite{petruvsevski2021odd} for the history, references, and discussions on many related problems, ranging from theoretical problems on graphs and multigraphs to problems involving algorithms.

As an initiating result in this area, Pyber~\cite{pyber1991covering} proved the following theorem. 
\begin{theorem}\label{Pyber_4} (\cite{pyber1991covering}) For any
simple graph $G$, $\chi_{\text{odd}}'(G) \leq 4$.
If $G$ is a connected graph (not necessarily a simple graph) and it has an even order, 
then $\chi_{\text{odd}}'(G)\leq 3$. 
\end{theorem}
It was also observed in~\cite{pyber1991covering} that the upper bound of $4$ in the previous result can be achieved by 
the wheel graph $W_4$. Lu\v{z}ar, Petru\v{s}evski, and {\v{S}}krekovski~\cite{luvzar2014odd} showed that connected loopless (multi)graphs are odd 5-edge-colorable, with one particular exception that is odd 6-edge-colorable. Before proceeding, we describe the exception(s). Given a loopless graph $G$, a \textit{bouquet} $\mathcal{B}$ is the set of parallel edges on a pair of adjacent vertices. In particular, for adjacent vertices $v,w\in V(G)$, the \textit{$vw$-bouquet} $\mathcal{B}_{vw}$ is the intersection  $E_G(v)\cap E_G(w)$. A \textit{Shannon triangle} is a loopless graph on three pairwise adjacent vertices.  If $p,q,r$ are parities of the sizes of its bouquets in non-increasing order, with $2$ (resp. $1$) denoting an even-sized (resp. odd-sized) bouquet, then $S$ is a Shannon triangle of \textit{type $(p,q,r)$}, denoted $S_{p,q,r}$. Examples can be found in Figure~\ref{fig:fourshannons}.

\begin{figure}[ht!]
	$$
		\includegraphics{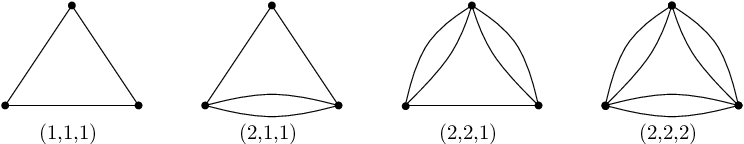}
	$$
	\caption{Examples of Shannon triangles (the smallest one of each type).}
	\label{fig:fourshannons}
\end{figure}

Noting that the edge set of every odd subgraph of a Shannon triangle is fully contained in a single bouquet, it can be verified that 
$\chi_{\text{odd}}'(S_{p,q,r})=p+q+r=\text{cov}_{\text{odd}}(S_{p,q,r})$. In particular, $\chi_{\text{odd}}'(S_{2,2,2})=6$ and $\chi_{\text{odd}}'(S_{2,2,1})=5$. In 2017, Petru\v{s}evski~\cite{petruvsevski2018odd} proved that the above two families of graphs are the only exceptions - 
all other connected loopless graphs are odd $4$-edge-colorable. Furthermore, Petru\v{s}evski~\cite{petruvsevski2018odd} also proved that when $\chi_{\text{odd}}'(G)=4$, a color class can be reduced to a size of at most 2. In the end of the paper, the author raised a natural question regarding the existence of an edge-analogue of the second part of Theorem ~\ref{Pyber_4}. That is, for a connected graph that 
is odd $4$-edge-chromatic, is it always possible to delete a particular edge (instead of a vertex), such that the resulting graph is odd $3$-edge-colorable? 
Additionally, the following conjecture was formally raised later in \cite{petruvsevski2021odd}.

\begin{conjecture}\label{conjecture_deletion}
    (\textbf{Conjecture 2.14 of~\cite{petruvsevski2021odd}})
For any connected loopless graph which is not a Shannon triangle of type $(2,2,2)$ or type $(2,2,1)$,
either $\chi_{\text{odd}}'(G) \leq 3$ 
or there exists a special edge $e$ whose removal will make the graph odd $3$-edge-colorable.
\end{conjecture}
 A considerable effort  towards the above conjecture was already made in~\cite{petruvsevski2018odd}. Note that an equivalent way to express the conjecture is that, assuming $\chi_{\text{odd}}'(G) =4$, there exists an odd 
 $4$-edge-coloring in which the color $4$ is used only once. 
 Another related result, presented in~\cite{petruvsevski2023odd}, provided a partial answer by 
 confirming the conclusion 
 for the special class of subdivisions of odd graphs. 
Now we are ready to state the first result of this paper, which fully resolves the above conjecture positively. 
\begin{theorem}\label{main} 
Any connected loopless graph $G$ that is not a Shannon triangle of type $(2,2,2)$ or $(2,2,1)$ is either odd $3$-edge-colorable or admits an odd $4$-edge-coloring with color set $\{1,2,3,4\}$ such that the color $4$  is used only on some edge $e$ whose endvertices are incident with a common color $c\in\{1,2,3\}$. Clearly, in the latter case, 
$\chi_{\text{odd}}'(G\backslash \{e\}) = 3$.
\end{theorem}

Now we move to graph coverability. Pyber also inquired in~\cite{pyber1991covering} 
whether three colors are sufficient to give an odd edge-covering for any simple graph $G$, a question later confirmed  by M\'atrai~\cite{matrai2006covering}.
\begin{theorem}(\cite{matrai2006covering})\label{Matrai_3} 
For any connected simple graph $G$, $\text{cov}_{\text{odd}}(G)\leq 3$.
\end{theorem}

This result was later generalized by Petru{\v{s}}evski and {\v{S}}krekovski~\cite{petruvsevski2019coverability} to multigraphs, by characterizing which multigraphs are odd $3$-edge-coverable. In addition, as a connection between $\chi_{\text{odd}}'(G)$ and $\text{cov}_{\text{odd}}(G)$, as well as a tool for obtaining more structural information about edge colors and edges, Petru{\v{s}}evski and {\v{S}}krekovski~\cite{petruvsevski2019coverability} defined the overlapping function. In more detail, 
for a $3$-edge-covering $\phi: E(G)\to 2^{[k]}\backslash \emptyset$, 
let $\text{ovl}(\phi)$ be the sum of sizes of all images under $\phi$ that are not singletons. The following are some examples of the values of the function $\text{ovl}(\phi)$:
$$\text{ovl}(\phi) =
  \begin{cases}
    0, & \text{ if } \phi \text{ is an edge-coloring.} \\
    2, & \text{ if } \phi \text{ has exactly one edge receiving two colors, with others receiving one color.} \\
    3, 
    & \text{ if } \phi \text{ has exactly one edge receiving three colors, with others receiving one color. } \\
    4, & \text{ if }\phi \text{ has exactly two edges receiving two colors each, with others receiving one color.}
  \end{cases}
$$

Given a graph $G$ with $\text{cov}_{\text{odd}}(G)\leq 3$, 
define 
\begin{equation}
\min \text{ovl}(G)= \min_{\phi} \text{ovl}(\phi),
\end{equation}
where the minimum is taken over all odd 3-edge-coverings $\phi$ of $G$. Immediately from the above definitions, 
if $G$ is odd 3-edge-colorable, then $\min \text{ovl}(G)=0$. 
As the main result in~\cite{petruvsevski2019coverability}, it was proved that except for a list of multigraphs with exactly $3$ vertices, for any other connected loopless graph $G$, one has $\text{cov}_{\text{odd}}(G)\leq 3$ and $\min \text{ovl}(G)\in \{0,2,3, 4\}$. In the same paper, it was also shown that there do exist multigraphs $G$ for which 
$\min \text{ovl}(G)$ is at least $4$. So this leads to the following natural conjecture.

\begin{conjecture}(\textbf{Conjecture 2.5 of~\cite{petruvsevski2021odd}})
    There is no simple graph $G$ such that $\min \text{ovl}(G)=4$. 
\end{conjecture}

Our second result of this paper strongly confirms the above conjecture.

\begin{theorem}\label{covering_conjecture}
    For any simple connected  graph $G$, we have $\min \text{ovl}(G)\leq 2$. 
\end{theorem}

The article is organized as follows. Section 2 is a brief overview of the notation and preliminary results. Subsequently, Section 3 addresses some important special cases in obtaining certain partial results. 
We then proceed to prove Theorem~\ref{main} in Section 4 and Theorem~\ref{covering_conjecture} in Section 5. In Section 6, we present some concluding remarks.

\section{Notation and Preliminary Results}

In this section, we introduce some notation that we will frequently use in this work.  We mention only those points of relevance that are liable to variations. The parameters $v(G)=|V(G)|$ and $e(G)=|E(G)|$ are \textit{order} and \textit{size} of $G$, respectively. A vertex $v$ is a \textit{cut vertex} of $G$ if $G-v$ has more (connected) components than $G$; otherwise, $v$ is an \textit{internal vertex} of $G$. 
Let $v$ be a cut vertex of $G$, and let $C_1,\ldots,C_k$ be the components of $G-v$. 
Then each $G[V(C_i)\cup \{v\}]$, for $i=1,\ldots,k$, is called a \textit{$v$-lobe} of $G$. 
 A \textit{block graph} is a connected graph without cut vertices.  Given a graph $G$, any maximal block subgraph $B$ is a \textit{block} of $G$. For a block $B$ of $G$, each vertex $v\in V(B)$ that is not a cut vertex of $G$ is an \textit{internal vertex} of $B$ (and of $G$); the collection of internal vertices of $B$ is denoted by $\mathrm{Int}_G(B)$.
Any block of $G$ that contains at most one cut vertex of $G$ is 
 an \textit{end-block} of $G$. Every vertex $v$ of a block graph $G$ has a neighbor among the internal vertices of each end-block of $G-v$.  

 \smallskip
 
Let $T$ be an even-sized subset of $V(G)$. A spanning subgraph $H$ of $G$ is a \textit{$T$-join} if $d_H(v)$ is odd for all $v\in T$ and even for all $v\notin T$.
The following lemma contains some long-known facts. 

\begin{lemma}\label{forest_coforest}
(\cite{schrijver2002combinatorial}) Given a connected graph $G$ and an even-sized subset $T$ of $V(G)$,
\begin{enumerate}
\item[$(1)$] there exists a $T$-join of $G$ that is a forest;
\item[$(2)$] there exists a $T$-join of $G$ that is a coforest;
\item[$(3)$] additionally, if $G$ is of even order then it contains a spanning odd coforest.
\end{enumerate}
\end{lemma}

\begin{lemma}\label{no_bridge} (Proposition 2.5 of ~\cite{petruvsevski2018odd})
If $G$ is a connected graph that contains a bridge, then $\chi_{\text{odd}}'(G)\leq 3$. 
\end{lemma}

\begin{lemma}\label{coforest}(Proposition 2.6 of ~\cite{petruvsevski2018odd})
    Let $G$ be a connected graph and $v$ be an internal vertex. Let $T\subseteq V(G)$ be an even-sized set. Then there exists a $T$-join $H$ of $G$ which is a coforest, such that the following additional property holds for $v$ in $\hat{H}$, where $\hat{H}=G\backslash E(H)$. 
    \begin{enumerate}
    \item[(a)] If $d_G(v)$ is odd and $v\in T$, or $d_G(v)$ is even and $v\notin T$, then $d_{\hat{H}}(v)=0$.
    \item[(b)] If $d_G(v)$ is odd and $v\notin T$, or $d_G(v)$ is even and $v\in T$, then $d_{\hat{H}}(v)=1$.
\end{enumerate}      
Moreover, in case (b), we can indicate a specific edge $e$ incident to $v$ such that $e\in E(\hat{H})$. 
\end{lemma}

Given an edge-coloring or edge-covering $\varphi$ of graph $G$, for any color $c$ let $E_c$ denote the \textit{color class of $c$}, that is, $E_c=\varphi^{-1}(c)$. The spanning subgraph of $G$ with edge set $E_c$ is denoted by $G_c$. The \textit{color degree} of $c$ at a vertex $v$ is the degree $d_{G_c}(v)$ of $v$ with respect to the graph $G_c$. A color $c$ is \textit{odd} (respectively, \textit{even}) \textit{at $v$} if $v$ the color degree $d_{G_c}(v)$ is odd (respectively, even). A color $c$ \textit{appears at $v$ under $\varphi$} if $d_{G_c}(v)>0$. We say that $\varphi$ is \textit{odd} (respectively, \textit{even}) \textit{at $v$} if each color appearing  at $v$ is odd (respectively, even). Similarly, we say that $\varphi$ is odd (respectively, even) \textit{away from $v$} if $\varphi$ is odd (respectively, even) at every $w\in V(G)-v$; no assumptions about the parity of the color degrees at $v$ are made.

\begin{lemma}\label{forest}
(\cite{matrai2006covering} and \cite{petruvsevski2018odd}) Any forest $F$ is odd $2$-edge-colorable. 
Moreover, for a fixed vertex $v$ in $F$, any local coloring of $E_F(v)$ that uses at most two colors extends to a 2-edge-coloring of $F$ that is odd away from $v$. 
\end{lemma}

For the first statement above, 
we can start by considering any non-isolated vertex $v\in F$.
If $d_F(v)$ is odd, we can assign the color $1$ to $E_F(v)$.
If $d_F(v)$ is even, we can assign the colors $1$ and $2$ to $E_F(v)$ arbitrarily, ensuring that both colors are used an odd number of times. 
Then, one can greedily extend the above coloring to an odd $2$-edge-coloring of $F$, as there are no cycles. Let us remark in passing that the second part of Theorem~\ref{Pyber_4} follows easily from the above lemmas~\ref{forest_coforest}$(3)$ and~\ref{forest}.

\smallskip

\begin{figure}[ht!]
	$$
		\includegraphics[scale=0.9]{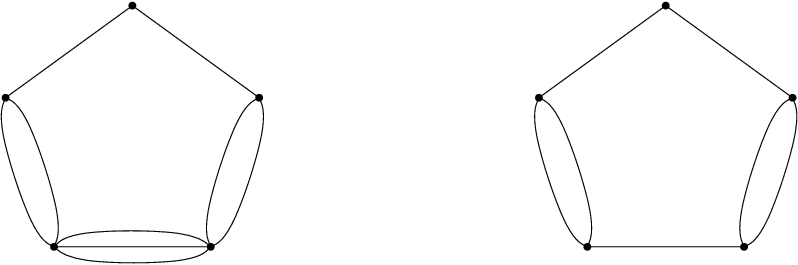}
	$$
	\caption{A graph $G$ (left) and its reduction $\mathrm{red}(G)$ (right), such that $\chi'_{\text{odd}}(G)=3$ and $\chi'_{\text{odd}}(\mathrm{red}(G))=4$.}
	\label{fig:Sereni}
\end{figure}

\begin{figure}[ht!]
	$$
		\includegraphics[scale=0.9]{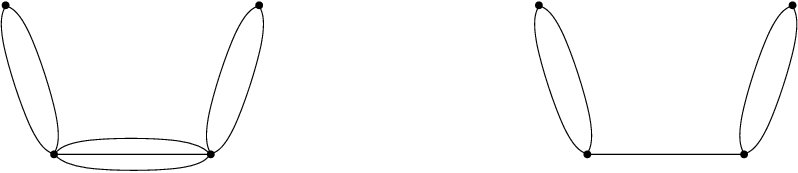}
	$$
	\caption{A graph $G$ (left) and its reduction $\mathrm{red}(G)$ (right), such that $\text{cov}_{\text{odd}}(G)=2$ and $\text{cov}_{\text{odd}}(\mathrm{red}(G))=3$.}
	\label{fig:counterexample}
\end{figure}

As defined in~\cite{luvzar2014odd}, the \textit{reduction} $\mathrm{red}(G)$ of a loopless graph $G$ is a spanning subgraph of $G$ where each odd-sized (respectively, even-sized) bouquet $\mathcal B$ is reduced to size $1$ (respectively, size 2).  Clearly, $\mathrm{red}(G)$ is of multiplicity at most $2$ and is uniquely determined by $G$. If $G=\mathrm{red}(G)$ then we call $G$ a \textit{reduced} graph. Observe that for every Shannon triangle $S$, $\mathrm{red}(S)$ is depicted in Fig.~\ref{fig:fourshannons} and is of the same type as $G$. The inequalities~\eqref{ineq}  below are easily shown.

\begin{equation}
         \label{ineq}
\chi_{\text{odd}}'(G)\leq \chi_{\text{odd}}'(\mathrm{red}(G)) \quad\text{ and }\quad\text{cov}_{\text{odd}}(G)\leq \text{cov}_{\text{odd}}(\mathrm{red}(G))
\end{equation}

Although both inequalities~\eqref{ineq} turn out to be equalities for every Shannon triangle $S$, there exist loopless graphs $G$ for which these are strict inequalities (c.f. Fig.~\ref{fig:Sereni} in regard to $\chi_{\text{odd}}'$ and Fig.~\ref{fig:counterexample} in regard to $\text{cov}_{\text{odd}}$). Note that the latter example refutes a conjecture proposed at the end of~\cite{petruvsevski2018odd}. The notions of `reduction' and `reduced graph' are relevant to us here because it can be easily argued that it suffices to prove Theorem~\ref{main} for reduced graphs.

For each vertex $v\in V(G)$, let $N_G(v)$ be the set of neighbors of $v$ in $G$. In a reduced graph $G$, $N_G'(v)$ denotes the set of neighbors of $v$ through a bouquet of size $1$, and $N_G''(v)$ denotes the set of neighbors of $v$ through a bouquet of size $2$.

\section{Some Special Cases}\label{special_cases}
Aiming to prove our two main theorems, in this section, we present  some preliminary results dealing with special cases. 
Note that some important partial results will already be obtained in this section. For example, 
Theorem~\ref{main} will be proved under the condition that the graph $G$ admits at least one cut vertex. 
We divide this section into two subsections for edge-colorings and edge-coverings, respectively.
Certain arguments are based on the various coloring and covering schemes provided in~\cite{petruvsevski2018odd} and~\cite{petruvsevski2019coverability}, respectively. 

\subsection{Odd-Edge-Colorings}

Let us first discuss graphs all of whose cycles (if any) share a common vertex. We begin with a particular one of multiplicity $2$ - the so-called `double join' of a path and a vertex $v$. Its special feature is the flexibility regarding odd $2$-edge-colorings that are odd away from $v$. This flexibility in regard to the parity of color degrees at the vertex $v$ will be important at some point in the proof of Theorem~\ref{main}. 

\begin{lemma}\label{two_ways}
    Let $H$ be a graph of order $v(H)\geq 4$, and let $v$ be an even vertex such that $N_H(v)=N_H''(v)=V(H-v)$. Moreover, let  $H-v$ be a simple path. Then we can obtain 
    two distinct $2$-edge-colorings $\varphi_1,\varphi_2:E(H)\to \{1,2\}$ with the following features.
    \begin{enumerate}
        \item The coloring $\varphi_1$ is an odd $2$-edge-coloring of $H$; 
        \item The coloring $\varphi_2$ is a 
 $2$-edge-coloring of $H$ which is odd away from $v$, 
 and both colors $1$ and $2$ are even at $v$.
       \end{enumerate}
\end{lemma}

\begin{proof}
Let $\{e_1,e_2\}$,$\{e_3,e_4\}$$,\ldots,$$\{e_{2k-1},e_{2k}\}$ denote the bouquets that are incident to $v$, and let $P_{k-1}=H-v$ be a path whose edge set is denoted by $\{f_1,f_2,\ldots,f_{k-1}\}$, such that $e_{2i-1},e_{2i},e_{2i+1},e_{2i+2},f_i$ form a Shannon triangle $S_{2,2,1}$ for every $i=1,\ldots,k-1$. 

We define the first coloring $\varphi_1$ as follows.
Assign the color $1$ to the following edges. 
\begin{align*}
      &    \{f_1\} ,    \{f_m \big| \, m \text{ is even} , 2 \leq m \leq k-1\},\\
      & \{e_1,e_2,e_3\} ,         \text{ and } \{e_{2m+1}, e_{2m+2} \big| \, m \text{ is even}, 2\leq m \leq k-1  \}. 
                     \end{align*} 
                     The remaining edges are  colored with $2$. 

We define the second coloring $\varphi_2$ as follows. Assign the color $1$ to the following edges. 
\begin{align*}
& \{f_m \big| \, m \text{ is odd}, 1\leq m\leq k-1\}, \\
& \{e_1,e_2\}, \, \text{ and } \{e_{2m+1}, e_{2m+2} \big| \, m \text{ is odd}, 1\leq m\leq k-1\}.
\end{align*}
The remaining edges are colored with $2$.     
\end{proof}

Turning to general graphs all of whose cycles (if any) share a common vertex, the following easy result follows from Lemma~\ref{forest} by splitting $v$ into leaves.

\begin{lemma}\label{feedback_blow_up} (Proposition 2.3 of ~\cite{petruvsevski2018odd})
Let $H$ be a graph with a vertex $v$ such that $H-v$ is a forest. Then $H$ admits a $2$-edge-coloring that is odd away from $v$. Additionally, if $d_H(v)$ is odd, the coloring is such that the color $1$ (respectively, $2$) is odd (respectively, even) at $v$. Contrarily, if $d_H(v)$ is even then either both color degrees at $v$ are even or both are odd.
\end{lemma}

\begin{lemma}\label{feedback}
	Let $H$ be a graph with a vertex $v$ such that $H-v$ is a forest. Assume $d_H(v)$ is even and an odd-sized bouquet $\mathcal{B}$ is incident with $v$. Then $\chi_{\text{odd}}'(H)\leq3$. Moreover, if   $\chi_{\text{odd}}'(H)=3$ then there exists an edge $e\notin E_H(v)$ such that $e$ is adjacent to $\mathcal{B}$ and $\chi_{\text{odd}}'(H\backslash\{e\})=2$.	
\end{lemma}

\begin{proof}
Split $v$ into $d_H(v)$ leaf vertices. The obtained graph $H'$ is a forest. Note that the bouquets of $H$ that are incident with $v$ become stars within $H'$; the edges of each such star are some of the new pendant edges of $H'$ and together comprise the even-sized set $E(H')\backslash E(H)$. We distinguish between two possibilities.

\smallskip
  
\noindent {\textbf {Case 1.}} \textit{A component of $H'$ has an odd-sized intersection with $E(H')\backslash E(H)$.} We show that then $\chi_{\text{odd}}'(H)=2$. In view of Lemma~\ref{forest}, take an odd $2$-edge-coloring of $H'$ with color set $\{1,2\}$. If this yields an odd $2$-edge-coloring of $H$, we are done (as $d_H(v)$ is even and positive). Contrarily, both color degrees at $v$ are even, that is, an even number of edges in $E(H')\backslash E(H)$ are colored by $1$ and the remaining even number of edges in $E(H')\backslash E(H)$ are colored by $2$. Returning to the considered odd $2$-edge-coloring of $H'$, permute the colors $1$ and $2$ in a component of $H'$ that has an odd-sized intersection with $E(H')\backslash E(H)$. Now this yields an odd $2$-edge-coloring of $H$. 

\smallskip

\noindent {\textbf {Case 2.}} \textit{Every component of $H'$ has an even-sized intersection with $E(H')\backslash E(H)$.} Consider the component of $H'$, a tree $T$, that contains the odd-sized star $\mathcal{B}'$ originating from $\mathcal{B}$. Let $w\in N_H(v)$ be the central vertex of $\mathcal{B}'$, and let $v_1,v_2,\ldots,v_{2s-1}$ be the leaves of $\mathcal{B}'$ (arising from the splitting of $v$). Say $v_{2s},v_{2s+1},\ldots,v_{2t}$ are the rest of the vertices within $T$ that have also arisen from the splitting of $v$. Note that $v_{2s},v_{2s+1},\ldots,v_{2t}\notin N_T(w)$. Consider the non-empty set $E_T(w)\backslash\mathcal{B}'$. For any edge $e\in E_T(w)\backslash\mathcal{B}'$, the forest $T-e$ has two components, one of which fully contains $\mathcal{B}'$. Let $V_e$ be the intersection between $\{v_{2s},v_{2s+1},\ldots,v_{2t}\}$ and the component of $T\backslash e$ that does not contain $\mathcal{B}'$. Clearly, $\{V_e : e\in E_T(w)\backslash\mathcal{B}'\}$ is a partition of $\{v_{2s},v_{2s+1},\ldots,v_{2t}\}$. Since the set $\{v_{2s},v_{2s+1},\ldots,v_{2t}\}$ is odd-sized, there exists an edge $e\in E_T(w)\backslash\mathcal{B}'$ such that $V_e$ is odd-sized as well. Thus, $T\backslash e$ consists of two components each of which contains an odd-sized portion of $\{v_1,v_2,\ldots,v_{2t}\}$. Therefore, a component of  $H'\backslash\{e\}$ has an odd-sized intersection with $E(H'\backslash\{e\})\backslash E(H\backslash\{e\})$. By Case 1, $\chi_{\text{odd}}'(H\backslash\{e\})=2$. Finally, observe that $e\notin E_H(v)$ and $e$ is adjacent to $\mathcal{B}$. 
\end{proof}

\begin{coro}\label{even_vertex_odd-sized_bouquet}
Let $G$ be a connected graph with an internal even vertex $v$ that is incident to an odd-sized bouquet. Then $G$ is either odd $3$-edge-colorable or admits an odd $4$-edge-coloring with color set $\{1,2,3,4\}$ such that the color $4$  is used only on some edge $e$ whose endvertices are incident with a common color $c\in\{1,2,3\}$.
\end{coro}

\begin{proof}
In view of the second part of Theorem~\ref{Pyber_4}, we may assume that $v(G)$ is odd. Thus $G-v$ is a connected graph of even order. Take a spanning odd coforest $K$ of $G-v$, and color $E(K)$ by $1$. Consider the graph $H=G\backslash E(K)$. Since Lemma~\ref{feedback} applies to $H$, take an odd-$3$-edge-coloring of $H$ with color set $\{2,3,4\}$ such that the color class $E_4\subseteq\{e\}$, where $e$ is an edge as described in Lemma~\ref{feedback}. Noting that $e\in E(G-v)$, the color $1$ appears at both endvertices of $e$. 
\end{proof}

Next, we consider reduced graphs with some particular vertices. The following two lemmas are based on Corollary 3.3 and Lemma 3.4 in~\cite{petruvsevski2018odd}, respectively.

\begin{lemma}\label{cor3.3}
Let $G$ be a reduced connected graph, and let $v$ and $w$ be odd vertices such that $w\in N_G''(v)$. If $v$ is an internal vertex of $G$ and $w$ is an internal vertex of $G-v$, then there exists an edge $f$ such that $G \backslash \{f\}$ admits an odd $3$-edge-coloring with a common color appearing at both endvertices of $f$.
\end{lemma}

\begin{proof}
We prove that any edge from $\mathcal{B}_{vw}$ works. In view of the second part of Theorem~\ref{Pyber_4}, we may assume that $v(G)$ is odd.
Write $\mathcal{B}_{vw}=\{e,f\}$. Since $w$ is an internal vertex of $G-v$, 
by Lemma~\ref{coforest}, there is a spanning odd coforest $K$ in $G-v$ such that $d_{\hat{K}}(w)=0$, where $\hat{K} = (G-v)\backslash E(K)$. Assign color $1$ to $E(K)$. 
Note that $d_G(v)$ is odd, and the graph $H=G\backslash E(K)$ meets all requirements of
Lemma~\ref{feedback_blow_up}. Thus $H$ admits a $2$-edge-coloring with color set $\{2,3\}$,
 which is odd away from $v$ and at $v$
the color $2$ (respectively, $3$) is odd (respectively, even). 
Since $E_{H}(w)=\mathcal{B}_{vw}$, we can assume that $e$ is colored by $2$ and $f$ by $3$.
This induces a coloring of the graph $G\backslash \{f\}$, which is an odd $3$-edge-coloring, and color $2$ appears at both endvertices of $f$.  
\end{proof}

\begin{lemma}\label{lemma3.4}
    Let $G$ be a reduced connected graph such that every even vertex $u$ satisfies $N_G'(u)=0$. Assume there are
   three vertices $v,w,q\in V(G)$ that satisfy the following.  
    \begin{enumerate}
        \item $v$ is an even internal vertex of $G$;
        \item $q$ is an internal vertex of $G-v$;
        \item $v$ and $q$ are nonadjacent, and $w$ is adjacent to both $v$ and $q$.
    \end{enumerate}
    Then either $G$ admits an odd $3$-edge-coloring, 
    or there is an edge $e$ such that $G\backslash \{e\}$ admits an odd $3$-edge-coloring with a common color appearing at both endvertices of $e$. 
\end{lemma}

\begin{proof}
   We may assume that $v(G)$ is odd, otherwise $G$ is odd $3$-edge-colorable by the second part of Theorem~\ref{Pyber_4}. Let us distinguish between four cases based on the parities of $d_G(w)$ and $d_G(q)$.  We write 
   $\mathcal{B}_{vw}=\{f_1,f_2\}$ and specify one edge $e\in \mathcal{B}_{wq}$.

 \smallskip  
   
   \noindent {\textbf{Case 1.}} \textit{$d_G(w)$ is odd and $d_G(q)$  is even.} We show that then $\chi_{\text{odd}}'(G)\leq3$. 
   Since $q$ is an even internal vertex in $G-v$, according to Lemma~\ref{coforest}, there is a spanning odd coforest $H$ in $G-v$, such that $E_{\hat{H}}(q)=\{e\}$, where $\hat{H} = (G-v)\backslash E(H)$. 
   Note that $d_{\hat{H}}(w)$ is even, and $d_{\hat{H}}(w) \geq 2$ as $e$ is present. 
 By Lemma~\ref{feedback_blow_up}, $G\backslash (E(H)\cup \{f_1,e\})$ has a 2-edge-coloring $\phi$ with color set $\{1,2\}$ such that $\phi$ is odd away from $v$, and at $v$
 the color $1$ is odd whereas the color $2$ is even. Since the degree of $w$ in $G\backslash (E(H)\cup \{f_1,e\})$ is even and positive, it has received both colors $1$ and 
 $2$. Then we assign $\{e,f_1\}$ with the color 2 and $E(H)$ with the color 3.
 This procedure yields an odd 3-edge-coloring of $G$. 

\smallskip
 
  \noindent {\textbf{Case 2.}} \textit{$d_G(w)$ is even and $d_G(q)$  is odd.} Say $\mathcal{B}_{wq}=\{e,e'\}$.
  Note the graph $G\backslash \{e'\}$ has a similar property with Case $1$. 
  Therefore, a similar argument gives an odd $3$-edge-coloring of the graph $G\backslash \{e'\}$. The color assigned to the edge $e$ clearly appears at both endvertices of $e'$.

\smallskip  
  
\noindent {\textbf{Case 3.}} \textit{$d_G(w)$ is odd and $d_G(q)$ is odd.}  Choose $e'\in  E_G(q)\backslash E_G(w)$. 
Then $G\backslash \{e'\}$ admits an odd $3$-edge-coloring as it satisfies the requirements of Case 1. Note that the color $3$ appears at both endvertices of $e'$. 

\smallskip

   \noindent {\textbf{Case 4.}} \textit{$d_G(w)$ is even and $d_G(q)$ is even. }
   We again write $\mathcal{B}_{wq}=\{e,e'\}$. There exists a spanning odd coforest $H$ of $G-v$ satisfying $E_{\hat{H}}(q)=e$. Similar with Case 1, by Lemma~\ref{feedback_blow_up}, 
    we can give a 2-edge-coloring $\phi$ with color set $\{1,2\}$ to $G'=G\backslash (E(H)\cup \{e,f_1\})$ such that $\phi$ is odd away from $v$, and at $v$ the color $1$ is odd whereas the color $2$ is even. 
   Note that $d_{G'}(w)$ is odd, which implies that 
   $w$ receives only one color under $\phi$. 
   If $w$ receives the color $1$, then we assign the color $2$ to $f_1$, the color $3$ to $E(H)$ and delete the edge $e$, which completes an odd $3$-edge-coloring of $G\backslash\{e\}$; the color assigned to $e'$ appears at both endvertices of $e$. 
   If $w$ receives the color $2$, then we assign the color $2$ to both $f_1$ and $e$. It can be checked that this is an odd $3$-edge-coloring for $G$. \end{proof}

We would like to address
another very special case where $G$ is a `doubled' complete graph, that is, every two vertices of $G$ are adjacent via a bouquet of size $2$. 

\begin{lemma}\label{doubleK_t}
Let $G$ be the graph with order $v(G)=2\ell+1$, $\ell\geq 2$, such that for any two vertices $v,w\in V(G)$ there is a bouquet $\mathcal B_{vw}$ of size $2$.  Then, for any edge $e\in E(G)$, the graph $G\backslash \{e\}$ is odd $3$-edge-colorable. Obviously,  under any such edge-coloring, a common color appears at both endvertices of $e$.
\end{lemma}
\begin{proof} Let us
specify two vertices $x$ and $y$ in $V(G)$, and let $H= G- \{x,y\}$, which is a copy of a `doubled' $K_{2\ell-1}$.
We split $E(H)$ into a union of three edge sets (regarded as subgraphs):
 $H_1$, $H_2$, and an edge $e \in \mathcal B_{v_1v_2}$, where $v_1$ and $v_2$ are two fixed vertices in $H$.  
We require that the subgraph $H_1$ (respectively, $H_2$) 
has even degree at the vertex $v_1$ (respectively, $v_2$), 
and odd degree at all other vertices in $V(H)$. 
For example, we can take $H_1$ to be the union of a star $S_2$ centered at $v_1$ with two leaves $v_2$ and $v_3$, where $v_3$ is any other vertex of $H$, 
together with a perfect matching in $H-\{v_1,v_2,v_3\}$. 
Then $H_2 =H \backslash (E(H_1) \cup \{e\})$. 
We then define the graph $H_1'$ to be the union of $H_1$ 
and all the bouquets $\mathcal B_{xw}$ for $w\in V(H)-v_1$ and one of the edges from the bouquet $\mathcal B_{xv_1}$. We can observe that $H_1'$ is an odd graph. 
Similarly, we define $H_2'$ to be  
 a union of $H_2$ and all the bouquets $\mathcal B_{yw}$ for $w\in V(H)-\{v_2\}$, along with one of the edges from the bouquet $\mathcal B_{yv_2}$. 
 The remaining graph $H_3=G\backslash (E(H_1')\cup E(H_2')\cup\{e\})$ is the union of $\mathcal B_{xy}$ and one edge from $\mathcal B_{xv_1}$ and one edge from $\mathcal B_{yv_2}$, which is an odd subgraph. 
 We assign  
 color $1$ to $H_1'$, 
 color $2$ to $H_2'$, and color $3$ to $H_3$.
  This is an odd $3$-edge-coloring of $G\backslash \{e\}$. 
\end{proof}


\subsection{Odd-Edge-Coverings} 

\begin{lemma}\label{cover_odd_component}
    Let $G$ be a simple connected graph of odd order. 
    If there exists a cut vertex $v$ of $G$ such that $G-v$ has a connected component of odd order, then $\min \text{ovl}(G)\leq 2$. 
\end{lemma}
\begin{proof}
  Let $C$ be a connected component of $G-v$ with $v(C)$ odd, and let $D=G-(V(C)\cup \{v\})$. 
 By Lemma~\ref{forest_coforest}, there are odd spanning coforests $F_1$ and $F_2$ in $G[V(C)\cup \{v\}]$ and $G[V(D)\cup \{v\}]$ respectively. 
 We also take them to have maximum sizes among all such odd spanning coforests.
Next, we obtain $T_1=G[V(C)\cup \{v\}]\backslash E(F_1)$ and $T_2=G[V(D)\cup \{v\}]\backslash E(F_2)$. 
  We assign colors $1$ and $2$ to the edges of $F_1$ and $F_2$, respectively.

  \smallskip
      
\noindent {\bf Case 1.} \textit{$d_{T_1\cup T_2}(v)\leq 1$.}
We assume that $d_{T_1}(v)\leq 1$ and $d_{T_2}(v)=0$; the reasoning is analogous if $d_{T_1}(v)= 0$ and $d_{T_2}(v)\leq 1$. Then, we proceed to provide an odd $2$-edge-coloring for 
$T_1$ with color set $\{2,3\}$, ensuring that $E_{T_1}(v)$ does not receive color 2. Subsequently, we give an odd $2$-edge-coloring for $T_2$ with color set  $\{1,3\}$. As a result, $\chi'_{\text{odd}}(G)=3$, that is, $\min \text{ovl}(G)=0$. 

\smallskip

\noindent {\bf Case 2.} \textit{$d_{T_1\cup T_2}(v)\geq 2$.}
If $d_G(v)$ is odd, then $d_{T_1\cup T_2}(v)$ is also odd. Consequently, we can derive an odd $3$-edge-coloring for $T_1\cup T_2$, consistent with the coloring of $F_1$ and $F_2$ provided earlier. Initially, we ensure that all edges incident to $v$ are assigned color $3$. Subsequently, for the forest $T_1-v$, we utilize Lemma~\ref{forest} to apply an odd $2$-edge-coloring with color set $\{2,3\}$.
Similarly, for the forest $T_2-v$, we obtain an odd $2$-edge-coloring with color set $\{1,3\}$. This completes an odd $3$-edge-coloring for the entire graph $G$, proving that $\min \text{ovl}(G)=0$. 

If $d_G(v)$ is even, we define an odd $3$-edge-covering $\varphi$ 
such that $\text{ovl}(\varphi)=2$. Let $u\in F_1$ be a vertex adjacent to $v$.
We assert the following.

\smallskip

\noindent{\bf Claim.} \textit{$T_1\cup T_2 \cup \{uv\}$ contains no cycle.}
Indeed, if $T_1\cup T_2\cup\{uv\}$ contains a cycle, the cycle must be contained in the $v$-lobe $G[V(C)\cup \{v\}]$.
Thus we can assume it has 
the form $vuw_1\cdots w_t$.
Then, we can update $F_1$ by removing the edge $vu$ and adding the path $w_1\cdots w_t$,
to obtain an odd coforest in $G[V(C)\cup \{v\}]$ which has larger size. This contradicts the maximum choice of $F_1$.  \hfill$\diamond$

\smallskip

Now, by Lemma~\ref{forest}, there exists an odd $2$-edge-coloring of 
$T_1\cup T_2\cup \{uv\}$ with  color set $\{3,4\}$, 
such that all the edges incident to $v$ are colored with $3$ since $d_{T_1\cup T_2\cup \{uv\}}(v)$ is odd. Subsequently, we replace all edges colored with $4$ in $C$ (respectively, $D$) with color $2$ (respectively, color $1$). This gives 
an odd $3$-edge-covering of $G$, where $uv$ is the only edge that receives two colors, namely, 
colors $1$ and $3$. Consequently, $\min \text{ovl}(G)\leq 2$. 
\end{proof}

Below we rephrase Claims 2 and 3 from the proof of Theorem 1 in~\cite{petruvsevski2019coverability}, only for simple graphs in our context. 
\begin{lemma}\label{claim23}
Let $G$ be a simple connected graph of odd order. Let $v$ be an even internal vertex of $G$ and $w\in N_G(v)$. We have the following:
\begin{enumerate} 
    \item[(a)] 
If $d_G(w)$ is even and $w$
is an internal vertex of $G-v$, then $G$ admits 
an odd $3$-edge-coloring, which implies $\min \text{ovl}(G)=0$.
    \item[(b)] 
If $d_G(w)$ is odd, and there exists
an internal vertex $u$ of $G-v$, such that $u\in N_{G-v}(w)$ and $u\notin N_G(v)$, then there is an odd $3$-edge-covering  $\varphi$ of $G$, 
with $\text{ovl}(\varphi)\leq 2$.
\end{enumerate}
\end{lemma}
\begin{proof} Let $d_G(v)=2k$. 

\smallskip

{\noindent \bf (a)} Since $w\in N_G(v)$ and $d_G(w)$ is even, we observe that $d_{G-v}(w)$ is odd. As $w$ is an internal vertex of the connected graph $G-v$, there is an odd factor $H$ of $G-v$, which is a coforest and contains the entire $E_{G-v}(w)$. Let $\hat H$ be the edge-complement of $H$ in $G \backslash \{vw\}$. Note that $H-v$ is acyclic.
Therefore, by Lemma~\ref{feedback_blow_up}, since $d_{\hat H}(v)$ is odd, there is a $2$-edge-coloring $\varphi:E(\hat H)\to \{1,2\}$ that is odd away from $v$, and at $v$ the color $1$ is odd whereas the color $2$ is even. We assign color $2$ to the edge $vw$ and color $3$ to all of $E(H)$. This completes an odd $3$-edge-coloring of $G$.

\smallskip

{\noindent \bf (b)} We distinguish between two cases depending on the parity of $d_G(u)$. 

\smallskip

\noindent \textbf {Case $1$.} \textit{$d_G(u)$ is even.} In $G-v$, by Lemma~\ref{coforest}(b), we can find an odd factor $H$, which is a coforest and contains $E_G(u)\backslash \{uw\}$. We assign color $3$ to $E(H)$ and let $\hat H$ be the edge-complement of $H$ in $G\backslash \{vw,uw\}$. 
By Lemma~\ref{feedback_blow_up}, since $d_{\hat H}(v)$ is odd, there is an edge-coloring $\varphi: E(\hat H)\to \{1,2\}$ that is odd away from $v$, and at $v$ the color $1$ is odd whereas the color $2$ is even. If $w$ is an isolated vertex in $\hat H$, then we assign the color $2$ to the edge $vw$ and the color $1$ to the edge $uw$. Contrarily, since $d_{\hat H}(w)$ is even, the edge set $E_{\hat H}(w)$ obtains both colors $1$ and $2$. We then assign the color $2$ to both edges $vw$ and $uw$. This results in an odd $3$-edge-coloring of $G$, hence $\min \text{ovl}(G)=0$.

\smallskip

\noindent \textbf {Case $2$.} \textit{$d_G(u)$ is odd.} Let us take an odd factor $H$ of $G-v$, which is a coforest and contains $E_G(u)$.
Assign the color $3$ to $E(H)$ and let $\hat H$ be the edge-complement of $H$ in $G\backslash \{vw,uw\}$. Again, 
by Lemma~\ref{feedback_blow_up}, since $d_{\hat H}(v)$ is odd, there is an edge-coloring $\varphi: E(\hat H)\to \{1,2\}$ that is odd away from $v$, and at $v$ the color $1$ is odd whereas the color $2$ is even. 
Noting that $d_{\hat H}(w)$ is odd, the edge set $E_{\hat H}(w)$ obtains exactly one of the colors $1$ or $2$.  If it obtains the color $1$, then we assign the color $2$ to the edge $vw$, resulting in an odd $3$-edge-coloring $\varphi$, that is, $\text{ovl}(\varphi)=0$.
Contrarily, if $E_{\hat H}(w)$ obtains the color $2$, then we color with $2$ both edges $vw$ and $uw$; thus the edge $uw$ receives two colors in this case. In either case, we have produced an odd $3$-edge-covering $\phi$ of $G$, with $\text{ovl}(\phi)\leq 2$. 
\end{proof}
The next corollary rephrases
Claim 7 from the proof of Theorem 1 in~\cite{petruvsevski2019coverability}.
\begin{coro}\label{connectivity_at_least_3}
Let $G$ be a simple connected graph of odd order. 
If $v$ is an even internal vertex of $G$ and $G-v$ contains no cut vertices, then $\min \text{ovl}(G)\leq 2$.  
\end{coro}
\begin{proof} 
We argue by contradiction. Let $w$ be a neighbor of $v$, which is an internal vertex of $G-v$ by assumption. 
Lemma~\ref{claim23} (a) implies that $d_G(w)$ is odd, and Lemma~\ref{claim23} (b) implies that every neighbor $u$ of $w$ must also be a neighbor of $v$. By repeatedly applying Lemma~\ref{claim23} and noticing that $G-v$ is connected, we conclude that every vertex of $G-v$ has odd degree in regard to $G$ and is adjacent to $v$. 
Now we list the set of common neighbors of $v$ and $w$ as 
$\{ q_1,\cdots,q_{2k}\}$. Assign the color $1$ to $wq_j$ and the color $2$ to $vq_j$ for each $j=1,\cdots, 2k$. Assign both colors $1$ and $2$ to $vw$. Finally, note that the remaining edges form an odd graph and assign the color $3$ to all of them. \end{proof}

\section{Proof of Theorem~\ref{main}}
This section is devoted to a proof of Theorem~\ref{main}. We will follow some ideas from~\cite{petruvsevski2018odd}. For brevity of presentation, it is convenient to introduce the ad hoc notion \textit{nice coloring} of a loopless graph $G$ defined as an odd edge-coloring with color set $\{1,2,3,4\}$ such that the color class $E_4$ satisfies the following two conditions:
\begin{enumerate}
\item[$(i)$] $|E_4|\in\{0,1\}$;
\item[$(ii)$] if $\mathcal{B}_{xy}\cap E_4\neq \emptyset$ then another common color (besides $4$) appears at $x$ and $y$.
\end{enumerate}

Thus, Theorem~\ref{main} states that, excluding the Shannon triangles $S_{2,2,2}$ and $S_{2,2,1}$, every other loopless connected graph $G$ admits a nice coloring.
Arguing by contradiction, suppose the above statement is false. Let us consider a counterexample $G$ that minimizes $m(G)$. In view of Theorem~\ref{Pyber_4}, the order $v(G)$ is odd. Throughout the proof, let $\mathcal{O}$ (respectively, $\mathcal{E}$) denote the set of odd (respectively, even) vertices of $G$. So $|\mathcal{O}|$ is even and $|\mathcal{E}|$ is odd. We provide several structural constraints for the supposed minimal counterexample $G$.

\smallskip

\noindent \textbf{Claim 1.} \textit{$G$ is a reduced graph.}

Otherwise, $\mathrm{red}(G)\subsetneq G$. By the minimality choice of $G$, there exists a nice coloring $\varphi$ of $\mathrm{red}(G)$. Extend $\varphi$ to $E(G)$ by applying the following at any bouquet $\mathcal{B}_{xy}$ of $G$ with size greater than $2$: if $\mathcal{B}_{xy}\cap E(\mathrm{red}(G))$ is a singleton colored by $4$, then assign to $\mathcal{B}_{xy}\backslash E(\mathrm{red}(G))$ a different color that appears at both $x$ and $y$ under $\varphi$ (such a color exists by the assumption $(ii)$); contrarily, if $\mathcal{B}_{xy}\cap E(\mathrm{red}(G))$ is not a singleton colored by $4$, then assign to $\mathcal{B}_{xy}\backslash E(\mathrm{red}(G))$ the smallest color appearing in $\mathcal{B}_{xy}\cap E(\mathrm{red}(G))$. This furnishes a nice coloring of $G$, a contradiction. 
\hfill$\diamond$

\smallskip

Next, we show that there are no cut vertices in $G$.

\smallskip

\noindent \textbf{Claim 2.} \textit{$G$ is a block graph.}

Arguing by contradiction, let $s$ be a cut vertex of $G$. We show that:
\begin{enumerate}
\item[$(1)$] every component of $G-s$ has an even order;
\item[$(2)$] for every $s$-lobe, the corresponding degree of $s$ is even; 
\item[$(3)$] $s$ is the only cut vertex of $G$, and there are only two blocks.
\end{enumerate}
First, we show $(1)$. Supposing the opposite, say $G_1$ is an $s$-lobe of $G$ such that $v(G_1)$ is even. Select an edge $e\in E_{G_1}(s)$. There exists a spanning odd coforest $H_1$ of $G_1$ such that $E_{\hat{H_1}}(s)\subseteq\{e\}$, where $\hat{H_1}=G_1\backslash E(H_1)$. Let $G_2=G-(G_1-s)$. Since the graph $G_2$ is connected and the order $v(G_2)$ is even, there exists a spanning odd coforest $H_2$ of $G_2$. We construct the following nice coloring of $G$: color $E(H_1)$ with $1$ (note that this introduces the color $1$ at both endvertices of $e$); take an odd edge-coloring of $\hat{H_1}\backslash E_{\hat{H_1}}(s)$ with color set $\{2,3\}$; if $E_{\hat{H_1}}(s)\neq \emptyset$ then color $e$ by $4$; assign the color $2$ to $E(H_2)$; color $E_{\hat{H_2}}(s)$ with $1$ (respectively, $3$) if it is even-sized (respectively, odd-sized); extend this coloring of $E_{\hat{H_2}}(s)$ to an edge-coloring of $\hat{H_2}$ with color set $\{1,3\}$ that is odd away from $s$. The obtained contradiction settles $(1)$.

Now we show $(2)$. Suppose the opposite, say $G_1$ is an $s$-lobe such that $d_{G_1}(s)$ is odd. In view of $(1)$, $v(G_1)$ is odd as well. Since $s$ is an internal vertex of $G_1$, there is a spanning odd coforest $H_1$ of $G_1-s$. Letting $F_1=G_1\backslash E(H_1)$, note that $F_1-s$ is a forest and that the degree $d_{F_1}(s)=d_{G_1}(s)$ is odd. Hence, $F_1$ admits an edge-coloring with color set $\{1,2\}$ that is odd away from $s$ and at $s$ the color $1$ (respectively, $2$) is odd (respectively, even). Combine such an edge-coloring of $F_1$ with a monochromatic coloring of $E(H_1)$ by $3$. Observe that this gives an edge-coloring $\varphi_1$ of $G_1$ with color set $\{1,2,3\}$ that is odd away from $s$ and at $s$:
\begin{itemize}
\item the color $1$ is odd;
\item the color $2$ is even (possibly non-appearing);
\item the color $3$ does not appear.
\end{itemize}
Define $G_2=G-(G_1-s)$. If $d_{G_2}(s)$ is odd, then take a $(V(G_2)-s)$-join $H_2$ of $G_2$ that is a coforest; thus, $G_2\backslash E(H_2)$ is a forest and $s$ is an odd vertex in this forest. Apply an odd edge-coloring to $G_2\backslash E(H_2)$ with color set $\{2,3\}$ such that $E_{G_2\backslash E(H_2)}(s)$ is monochromatically colored by $2$. Assign the color $1$ to $E(H_2)$. This completes an odd $3$-edge-coloring of $G$. Consequently, it must be that $d_{G_2}(s)$ is even. Moreover, if the color $2$ does not appear at $s$ under $\varphi_1$, then (even if $d_{G_2}(s)$ happens to be even) the same coloring of $E(G_2)$ combines with $\varphi_1$ to produce an odd $3$-edge-coloring of $G$. So, the color $2$ appear (and is even) at $s$ under $\varphi_1$.

Next, we argue that $N''_{G_2}(s)=\emptyset$. Supposing the opposite, let $u\in N''_{G_2}(s)$. Define $T=\{x\in (V(G_2)-s): d_G(x) \text{ is even}\}$. Since $T$ is even-sized, let $F_2$ be a forest that is a $T$-join of $G_2$. The only even vertex of $\hat{F_2}=G_2\backslash E(F_2)$ is $s$. It must be that $d_{F_2}(s)=0$, for otherwise an odd $3$-edge-coloring of $G$ would be obtained by combining $\varphi_1$ with a monochromatic coloring of $E(\hat{F_2})$ by the color $1$ and an odd edge-coloring of $F_2$ with color set $\{2,3\}$ under which the color $2$ appears at $s$. However, then we produce a nice coloring of $G$ as follows: color $\mathcal{B}_{su}$ by $2$ and $4$, assign the color $1$ to $E(\hat{F_2})\backslash \mathcal{B}_{su}$, apply an edge-coloring to $F_2$ with color set $\{2,3\}$ that is odd away from $u$ and colors $E_{F_2}(u)$ monochromatically with $2$ (respectively, $3$) if it is even-sized (respectively, odd-sized), and combine with $\varphi_1$. The obtained contradiction proves that $E_{G_2}(s)$ induces a star.

Now take a spanning odd coforest $H_2$ of $G_2-s$. Let $\hat{H_2}=(G_2-s)\backslash E(H_2)$. Note that $\hat{H_2}$ is a forest (within $G_2-s$). If all of $N_{G_2}(s)$ are isolated vertices in $\hat{H_2}$, then we complete a nice coloring of $G$ as follows: color a selected edge $e\in E_{G_2}(s)$ by $4$; color $E_{G_2}(s)\backslash\{e\}$ with $2$; assign $E(H_2)$ with the color $1$; take an odd edge-coloring of $\hat{H_2}$ with color set $\{2,3\}$, and combine with $\varphi_1$. Contrarily, if there is a vertex $u\in N_{G_2}(s)$ such that $d_{\hat{H_2}}(u)>0$, then color the edge $su$ by $4$, $E(H_2)$ by $3$, use an edge-coloring of $G_2\backslash(E(H_2)\cup\{su\})$ with color set $\{1,2\}$ that is odd away from $s$ and under which the color $2$ is odd at $s$. This completes a nice coloring of $G$. Indeed, both colors $1$ and $2$ appear at $s$ and at least one of them appears at $u$. The obtained contradiction settles $(2)$.

Finally, let us show $(3)$. Consider an end-block $B$ of $G$, and let $s$ be the cut vertex of $G$ contained in $V(B)$. So $s$ is an internal vertex of $B$, the order $v(B)$ is odd by $(1)$, and $d_B(s)$ is even by $(2)$. Consequently, there exists a $(V(B)-s)$-join $H$ of $B$ that is a coforest and uses all of $E_B(s)$. Supposing there are more than $2$ blocks of $G$, let $G'=G-(B-s)$. Then $G'$ is surely connected but not a block graph. Therefore, by the minimal choice of $G$, there is a nice coloring $\varphi'$ of $G'$. Take $c\in\{1,2,3\}$ to be a color that appears at $s$ under $\varphi'$. Color $E(H)$ by $c$ and take an odd edge-coloring of $\hat{H}=B\backslash E(H)$ with color set $\{1,2,3\}-\{c\}$. This completes a nice coloring of $G$, a contradiction that settles $(3)$.

Now then, in view of Figure~\ref{fig:nice}, it cannot be that both blocks of $G$ come from $\{S_{2,2,2},S_{2,2,1}\}$. Consequently, there is an end-block $B$ of $G$ such that the other end-block  $G'=G-(B-s)\notin\{S_{2,2,2},S_{2,2,1}\}$. However, then the nice coloring constructed in the demonstration of $(3)$ yields a final contradiction. 
\hfill$\diamond$
\begin{figure}[ht!]
	$$
		\includegraphics[scale=0.35]{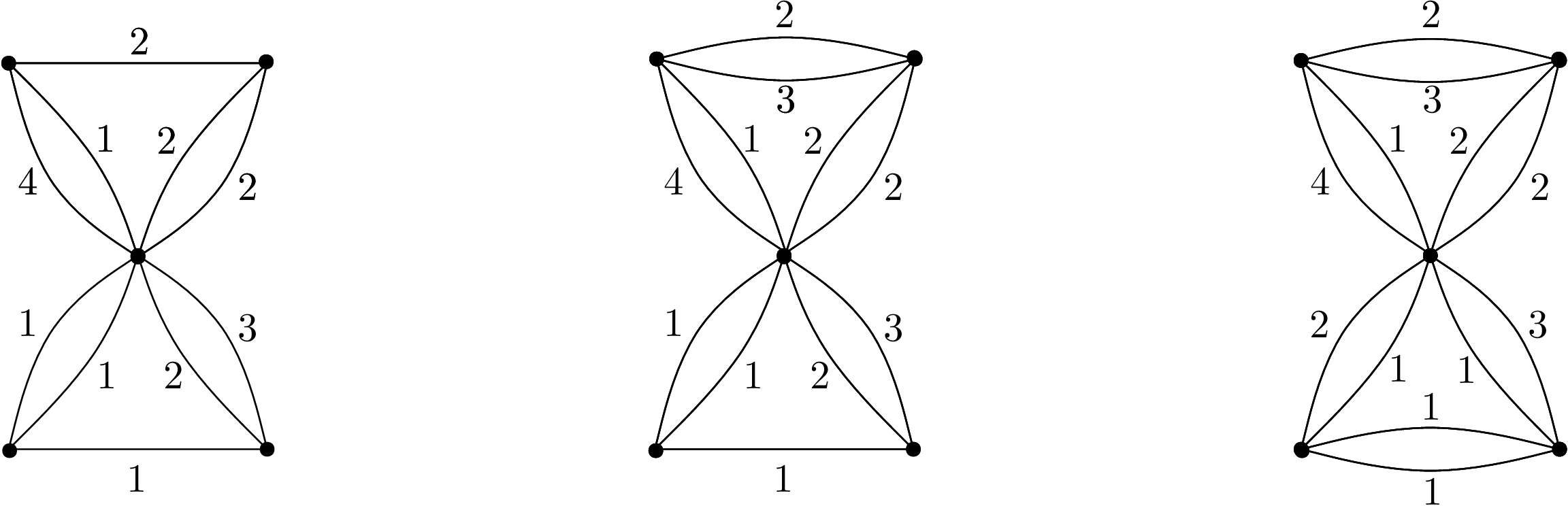}
	$$
	\caption{Nice colorings of three graphs, each comprised of two block that are $S_{2,2,2}$ or $S_{2,2,1}$.}
	\label{fig:nice}
\end{figure}
\smallskip

Let us draw some conclusions. In view of Corollary~\ref{even_vertex_odd-sized_bouquet}, Claim~2 implies that every vertex $v\in \mathcal{E}$ satisfies that $N_{G}(v)=N_{G}''(v)$. We have already deduced that $v(G)$ is odd. It must be that $v(G)\geq5$, for otherwise Claim~2 alone would imply that $G\in\{S_{1,1,1},S_{2,1,1}\}$, which is absurd as both $S_{1,1,1}$ and $S_{2,1,1,}$ contain an even vertex $v$ for which $N'_{G}(v)\neq\emptyset$. 

\smallskip

\noindent \textbf {Claim 3.} \textit{If $v\in \mathcal E$ and $B$ is an end-block of 
$G-v$, then $\text{Int}_{G-v}(B)\subseteq N_G(v)=N_G''(v)$. }

Since $G$ is a block graph, $\mathrm{Int}_{G-v}(B)\cap N_{G}(v)\neq \emptyset$. Hence, by Lemma~\ref{lemma3.4}, $\mathrm{Int}_{G-v}(B)- N_{G}(v)$. Indeed, supposing $\mathrm{Int}_{G-v}(B)- N_{G}(v)\neq\emptyset$, from the connectedness of $G[\mathrm{Int}_{G-v}(B)]$ it follows that the edge cut $E[\mathrm{Int}_{G-v}(B)\cap N_{G}(v),\mathrm{Int}_{G-v}(B)- N_{G}(v)]\neq\emptyset$, that is, a pair $w,q$ of adjacent vertices $w\in\mathrm{Int}_{G-v}(B)\cap N_{G}(v)$ and $q\in\mathrm{Int}_{G-v}(B)- N_{G}(v)$ exists. However, such a pair is impossible by Lemma~\ref{lemma3.4}.
\hfill$\diamond$

\smallskip

\noindent \textbf {Claim 4.}  \textit{There exists a vertex $v\in \mathcal{E}$ such that $G-v$ is not a block graph.} 

Suppose for a contradiction that for every vertex $v\in \mathcal E$ the graph $G-v$ is a block graph. 
By Claim 3, for each $v\in \mathcal{E}$, we have $N_G(v)=N_G''(v)=V(G)-v$. 

If $|\mathcal E|=1$, we first observe that every vertex in $G-v$ has odd degree. So there must exist a pair of vertices 
$u,w$ such that $|\mathcal B_{uw}|=1$. Denote $e=uw$, and below we will show that removing $e$ from $E(G)$ results in a graph $G'=G\backslash \{e\}$ that admits an odd $3$-edge-coloring. 
Since $G-v$ is a block graph, $G'-v$ is connected and there is a path $P$ in $G'-v$ connecting $u$ and $w$. 
We can see that $H= (G'-v)\backslash E(P)$ is an odd graph. 
Assign the color $3$ to $E(H)$. 
List the vertex set of $G-v-V(P)$ as $\{v_1,v_2,\ldots, v_m\}$ and write 
$\mathcal{B}_{vv_i}=\{e_{2i-1},e_{2i}\}$ for each $i=1,\ldots, m$. 
Then the edges of $\{e_i \big| \, 1\leq i\leq 2m, i \text{ is odd}\}$ are assigned with the color $1$ and 
$\{e_i\big| \, 1\leq i\leq 2m, i \text{ is even}\}$ are assigned with the color $2$. 

Observe that by excluding isolated vertices, the remaining graph of $ G'\backslash E(H)$ has exactly the same structure 
as described in 
Lemma~\ref{two_ways}. 
If $m$ is even,  we use the first $2$-edge-coloring from Lemma~\ref{two_ways}. 
Contrarily, if $m$ is odd, we use the second  $2$-edge-coloring from Lemma~\ref{two_ways}.
In either case, we eventually obtain an odd $3$-edge-coloring for $G'=G\backslash \{e\}$. Note that the color $3$ appears at both endvertices of $e$. Assigning the color $4$ to $e$ completes a nice coloring of $G$.

For the case where $|\mathcal{E}|\geq 3$, according to Lemma~\ref{doubleK_t}, $G$ must have at least two odd vertices. Moreover, there must exist 
two adjacent odd vertices, say $y$ and $z$, for which $|\mathcal B_{yz}|=1$.
Let us select a triplet of vertices $v, w, x$ from $\mathcal{E}$. 
Now, we denote $\mathcal{B}_{vy}=\{e',e''\}$ and $\mathcal{B}_{wz}=\{f',f''\}$. Let us define $T=V(G)-\{v,w,x,y,z\}$. Notice that $T$ is  
a set containing an even number of vertices.  
Below we will define a $T$-join $H$ of $G-\{v,w,x\}$ which is also a coforest. 

Specifically, let $p_1,\ldots, p_{2\ell}$ and $o_1,\ldots, o_{2k}$ represent the even and odd vertices in $T$, respectively. It is worth noting that $\ell$ and $k$ can be zero. When both $\ell$ and $k$ are zero, $H$ is an empty graph.
Otherwise, by definition, the edge set $E(H)$ is comprised from the following parts: 
\begin{enumerate}
 \item $E(G[\mathcal{O}])\backslash \{yz\}$;
\item All the bouquets $\mathcal{B}_{p_ip_i'}$, for $i, i'=1,\ldots, 2\ell$ and $i\neq i'$, $\mathcal{B}_{p_io_j}$, and $\mathcal{B}_{p_iz}$, for $i=1,\ldots,2\ell$, $j=1,\ldots, 2k$;
\item one of the edges in each bouquet $\mathcal{B}_{p_iy}$, for $i=1,\ldots, 2\ell$.
\end{enumerate}
We then denote $\hat{H}= (G-\{v,w,x\}) \backslash E(H)$ as the edge-complement of $H$ in $G-\{v,w,x\}$. 
In fact, by definition, $\hat{H}$ is an odd star centered at $y$, with the vertices
 $z$, $p_1,\ldots, p_{2\ell}$ being the leaves.  
 
 Consider the graph $G\backslash\{e''\}$. We assign the color $1$ to $E(H)\cup \{e',f'\}\cup \mathcal{B}_{vz}$, the color $2$ to $\mathcal{B}_{xz}\cup \{f''\}$, and the color $3$ to $E(\hat{H})$. 
 For each $p_i$, $i=1,\ldots, 2\ell$, assign the color $3$ to $\mathcal{B}_{p_ix}$
 and also to $\mathcal{B}_{yx}$. 
 For each $o_j$, $j=1,\ldots, 2k$, 
 the two edges in $\mathcal{B}_{o_jx}$ are assigned the colors $2$ and $3$, respectively. It can be verified that so far the colors $2$ and $3$ are even at $x$. 
 Next, we assign the two edges in $\mathcal{B}_{vx}
$ with the colors $2$ and $3$, respectively, 
assign both edges in $\mathcal{B}_{wx}$ with the color $2$, and assign all the remaining uncolored edges of $G$ with the color $1$. 
This yields an odd $3$-edge-coloring for $G\backslash\{e''\}$. Clearly, the color $1$ (and the color $3$) appears at both endvertices of $e''$. By assigning the color $4$ to $e''$, we complete a nice coloring of $G$. 
\hfill$\diamond$

\smallskip

According to Claim 4, let us choose $v\in \mathcal E$ such that $G-v$ is not a block graph. For the next three claims, which come from~\cite{petruvsevski2018odd}, we fix an end-block $B$ of $G-v$ and denote by $s$ the unique cut vertex of $G-v$ that belongs to $V(B)$. 

\smallskip

\noindent \textbf {Claim 5.} \textit{$\text{Int}_{G-v}(B)\subseteq \mathcal{O}$.}  

Suppose $w$ is an even vertex in $\text{Int}_{G-v}(B)$. 
Let $B'$ be another end-block of $G-v$, and $w'\in \text{Int}_{G-v}(B')$. 
Note that there is no edge incident to both $w$ and $w'$, and $v$ is a common neighbor of $w$ and $w'$. We apply Lemma~\ref{lemma3.4} to the triple $\{w,v,w'\}$ and conclude that $w'$ is a cut vertex of $G-w$, 
which means that $G-\{w,w'\}$ must be a disconnected graph. 
This has several strong implications. 
Firstly, $\text{Int}_{G-v}(B')=\{w'\}$. Secondly, by Claim 3, we conclude that $\text{Int}_{G-v}(B)=\{w\}$.
Finally, $N_G(v)=N''_G(v)=\{w,w'\}$. 
Recall that $s$ is the unique cut vertex of $G-v$ contained in $B$. 
Let us denote by $s'$ the unique 
cut vertex of $G-v$ contained in $B'$ (note that it is possible that $s=s'$). 
As a particular consequence of the above arguments, we see
$v$ and $s$ (or $s'$) are non-adjacent. 
Note that $|\mathcal{B}_{ws}|=2$, since $w$ is an even vertex.

In the graph $G-\{v,w,w'\}$, we apply Lemma~\ref{forest_coforest} 
to find an odd spanning coforest $H$, and write $\hat{H}=(G-\{v,w,w'\})\backslash E(H)$.
Choose an edge $e \in \mathcal{B}_{w's'}$. 
Assign the color $3$ to $E(H)$, and use the colors $1$ and $2$ to give $E(\hat{H}+e)$ an odd $2$-edge-coloring such that $e$ receives the color $1$. Assign the color $3$ to both edges in $\mathcal{B}_{sw}$, assign the color $1$ to both edges in
$\mathcal{B}_{vw'}$, and color $\mathcal{B}_{vw}$ by  $1$ and $3$.  If $|\mathcal{B}_{w's'}|=1$, then 
we have already obtained an odd $3$-edge-coloring of $G$. 
Otherwise, let us denote by $e'$ the other edge in $\mathcal{B}_{w's'}$ other than $e$. By assigning the color $4$ to $e'$ we complete a nice coloring of $G$, since the color used on $e$ (namely, the color $1$) appears at both endvertices of $e'$.\hfill$\diamond$

\smallskip

\noindent \textbf {Claim 6.} \textit{The graph $B-s$ is a simple graph.}
 
Suppose that we find vertices $w'$ and $w''\in \text{Int}_{G-v}(B)$ 
such that $|\mathcal B_{w'w''}|=2$. Note that by Claim 5, $w', w''\in \mathcal{O}$.
From Claim 3 we know that $v$ is adjacent to all vertices in $\text{Int}_{G-v}(B)$. And also $s$ and $v$ are connected in $G-\text{Int}_{G-v}(B)$. It follows that 
$G-w'$ and $G-\{w',w''\}$ are both connected.  
We then apply Lemma~\ref{cor3.3} to obtain a contradiction. \hfill$\diamond$

\smallskip

\noindent \textbf {Claim 7.} \textit{$s\in \mathcal O$ and $\text{Int}_{G-v}(B)\cap N_G''(s)=\emptyset$.} 
 
Suppose $s\in \mathcal E$. Lemma~\ref{lemma3.4} implies that $\text{Int}_{G-v}(B)\subseteq N_G(s)=N_G''(s)$.
Take any $w \in \text{Int}_{G-v}(B)$. Claims 3 and 6 imply that $N_G''(w)=\{v,s\}$.
Since $d_G(w)$ is odd, $N_G'(w)$ consists of an odd number of internal vertices of $B$, listed as $\{u_1,\ldots,u_{2\ell+1}\}$. For each $j=1,\ldots,2\ell+1$, choose one edge $e_j\in \mathcal B_{vu_j}$. Let $T$ be the set of all even vertices in the graph $G-v - (B-\{s\})$. Clearly, $T$ is even-sized and contains 
$s$. We apply Lemma~\ref{forest_coforest} 
to find a forest $F$ which is also a $T$-join.

Let $\mathcal B_{ws}=\{e,e'\}$. Consider the graph $G'=G\backslash \{e'\}$. 
In view of Lemma~\ref{forest}, assign the color $1$ to  $E_F(s)$ and extend  
to an odd $2$-edge-coloring of $F$. 
Then assign the color $1$ to $\{wu_1,\ldots, wu_{2\ell+1}\}$, and assign the color $2$ to $\{e_1,\ldots,e_{2\ell+1}\} \cup \mathcal B_{vw} \cup \{e\}$. The remaining uncolored edges of $G'$
form an odd graph; assign the color $3$ to all of them. 
Thus, an odd $3$-edge-coloring for $G'=G\backslash \{e'\}$ is obtained. Complete a nice coloring of $G$ by assigning the color $4$ to $e'$. This shows that $s\in \mathcal O$. 

Now, Claim 5 and Lemma~\ref{cor3.3} imply that $\text{Int}_{G-v}(B)\cap N_G''(s)=\emptyset. 
$\hfill$\diamond$

\smallskip

We are now ready to make one final claim that will lead to the completion of our proof. Remember, $v$ is an even vertex of $G$ for which $G-v$ is not a block graph.
Our next claim shows that there is an end-block of $G-v$  whose order is at least $4$.

\smallskip

\noindent \textbf {Claim 8.} \textit{There is an end-block $B$ of $G-v$ with $v(B)\geq 4$.}

Firstly, no end-block $B$ has order $1$ by the very definition of a cut vertex. Secondly, in view of Claims 5, 6, and 7, no end-block $B$ has order $3$ vertices; indeed, the supposition $v(B)=3$ would force $B=K_3$, which contradicts with the inclusion $\mathrm{Int}_{G-v}(B)\subseteq \mathcal{O}$. 
Suppose now that an end-block of $G-v$ contains exactly two vertices, say $u$ and $s$, 
where $u$ is an internal vertex of the block (and of $G-v$) and $s$ is a cut vertex of $G-v$. 
By Claims 5, 6, and 7, we know that $d_G(u)=3$, $d_G(s)\geq 3$ is odd,
and $|\mathcal B_{us}|=1$. 
Since $s$ is a cut vertex of $G-v$, it must be adjacent to some vertex other than $u$ and $v$ via an edge $e$. 
Below, we will show that $G'=G\backslash \{e\}$ admits an odd $3$-edge-coloring under which a common color appears at both endvertices of $e$. 

Note that the vertex $s$ is not a cut vertex in $G'$ and it is also not a cut vertex of $G'-u$. 
In the graph $G'-u$, since $d_{G'-u}(s)$ is odd, 
Lemma~\ref{coforest}(a) implies that there exists a spanning odd coforest $H$ 
such that $d_{\hat{H}}(s)=0$, where $\hat{H}=(G'-u)\backslash E(H)$.
 Since $d_{G'-u}(v)$ is positive and even, both $d_H(v)$ and $d_{\hat{H}}(v)$ are odd. 
Assign the color $3$ to $E(H)$, 
and take an odd $2$-edge-coloring of $\hat H$ with color set $\{1,2\}$ under which all of $E_{\hat{H}}(v)$ receives the color $1$. 
Finally, assign the color $1$ to $\mathcal{B}_{uv} \cup \{us\}$.  
This procedure results in an odd $3$-edge-coloring for $G'=G\backslash \{e\}$ under which the color $3$ appears at both endvertices of $e$. Thus, by assigning the color $4$ to $e$ we complete a nice coloring of $G$. 
\hfill$\diamond$

\medskip
  
\noindent {\bf End of the proof of Theorem~\ref{main}.} 
Let $B$ be an end-block of $G-v$ containing at least $4$ vertices, and $s\in B$ be a cut vertex. Note that $d_B(w)\geq 3$, for $w\in Int_{G-v}(B)$. Let $G_1=G-v$. Hence, by applying Lemma~\ref{coforest},
$G_1$ admits a spanning odd coforest $F$, and therefore $\hat{F}=G_1 \backslash E(F)$ is a forest. 
We \textbf{assert} that $F$ contains all edges in $B$. In other words,  the edge set of $\hat{F}$ does not intersect with the edge set of $B$. Suppose otherwise; then 
this intersection is again a forest and therefore has a leaf vertex in $B-s$. It means there is a vertex 
$w \in \text{Int}_{G-v}(B)$ such that $d_{\hat{F}}(w)=1$. However, this implies that $d_{F}(w)$ is even and positive, which is absurd. 
The assertion follows. 

By the previous claims, $B-s$ is a simple connected graph with minimum degree at least 2. Then $B-s$ contains a cycle $C$. Let $F'=F\backslash E(C)$ and note that $F'$ is a spanning odd coforest of $G_1\backslash E(C)$. 
Let $e$ be an arbitrary edge of $C$, and let $P=C\backslash \{e\}$. 
Recall that $V(P)\subseteq N''_G(v)$. We will prove that $G\backslash \{e\}$ admits an odd 3-edge-coloring under which a common color appears at both endvertices of $e$. 

Assign the color $3$ to $E(F')$ and take an odd $2$-edge-coloring of the forest $\hat{F}$ with color set $\{1,2\}$. 
Consider any vertex $w\in (N_G''(v)- V(P))$. We extend the odd coloring to the edges of $\mathcal{B}_{vw}$. If the edge set $E_{G_1}(w)$ has received 
only the color $3$, then color the bouquet $\mathcal B_{vw}$ (consisting of two edges) with  
$1$ and $2$. 
Contrarily, if the edge set $E_{G_1}(w)$ has already received the color $1$ (or the color $2$), we assign this color 
also to $\mathcal B_{vw}$.
The remaining uncolored edges are the edges of $P$ and the edges connecting $P$ and $v$. These form a graph $H$, which has exactly the same structure as described in
Lemma~\ref{two_ways}, leading to two possible cases.  
\begin{enumerate}
\item[Case $1$.] If both colors $1$ and $2$ appear an
 even number of times in 
the edge set $\bigcup_{w\in (N_G''(v)\backslash V(P))}\mathcal B_{vw}$, we use the first method in Lemma~\ref{two_ways} 
 to assign colors to $E(H)$.
 \item[Case $2$.] If both colors $1$ and $2$ appear an odd number of times in the edge set $\bigcup_{w\in (N_G''(v)\backslash V(P))}\mathcal B_{vw}$, 
we use the second method in Lemma~\ref{two_ways} to assign colors to $E(H)$. 
\end{enumerate}

 In either case, we can obtain an odd $3$-edge-coloring for the graph $G\backslash \{e\}$. Note that the color $3$ appears at both endvertices of $e$. By assigning the color $4$ to $e$ we complete a nice coloring of $G$, a contradiction.
\qed

\section{Proof of Theorem~\ref{covering_conjecture}}
This section is devoted to a proof of Theorem~\ref{covering_conjecture}. 
Our proof, conducted by contradiction, considers a possible minimal counterexample and employs the technique of induction.
The main inductive step is presented in the following proposition.
\begin{prop}\label{main_step}
Let $G$ be a simple connected graph with $\min \text{ovl}(G)\geq 3$, and let $v\in V(G)$ be an even vertex of $G$.  
\begin{enumerate}[(a)]
    \item If $v$ is a cut vertex of $G$, then there is a connected component $C$ of $G-v$ such that
$\min \text{ovl}(G-C)\geq 3$.
\item If $v$ is not a cut vertex of $G$,
then $G-v$ is not a block graph. Furthermore,
 there is an end-block $B$ of $G-v$ containing a unique cut vertex $s$ of $G-v$ 
such that $\min \text{ovl}(G-(B-s)) \geq 3$.
\end{enumerate}
\end{prop}
\begin{proof} Note that $v(G)$ is odd, as otherwise $\min \text{ovl}(G)= 0$ by Theorem~\ref{Pyber_4}. 

\smallskip

\noindent {\bf (a)} By Lemma~\ref{cover_odd_component}, we may assume that all connected components of $G-v$ are of even order. There are two cases to consider.

\smallskip

\noindent \textbf{Case 1.} \textit{For every connected component $C$ of $G-v$, the degree $d_{G[V(C)\cup \{v\}]}(v)$ is odd.} 
In this case, since $d_G(v)$ is even, there are an even number of components of $G-v$, say, $C_1,\cdots, C_{2k}$. In the graph $G[V(C_1)\cup \{v\}]$, by Lemma~\ref{coforest}(b), we can find a $V(C_1)$-join $H_1$ which is a coforest and  satisfies $d_{\hat H_1}(v)=1$, where $\hat H_1$ is the edge complement of $H_1$ in $G[V(C_1)\cup \{v\}]$.  
Assign the color $3$ to $E(H_1)$. Since $\hat H_1$ is a forest, take an odd $2$-edge-coloring of $\hat H_1$ with color set $\{1,2\}$ under which $E_{\hat H_1}(v)$ is entirely  colored by $1$. 

For each of the remaining connected components $C_j$, $j=2,\cdots, 2k$, we follow a similar process:
take a $V(C_j)$-join $H_j$ of the graph $G[V(C_j)\cup \{v\}]$ that is a coforest and color its edge set with $1$ . 
Let $\hat H_j= G[V(C_j)\cup \{v\}] \backslash E(H_j)$ be the edge-complement of $H_j$ in $G[V(C_j)\cup \{v\}]$. 
Apply an odd $2$-edge-coloring 
to $\hat H_j$ with color set $\{2,3\}$ such that $E_{\hat H_j}(v)$ obtains the color $3$.  

This furnishes an 
odd $3$-edge-coloring of $G$. Consequently, $\min \text{ovl}(G)=0$, a contradiction.

\smallskip

\noindent \textbf{Case 2.} \textit{There is a connected component $C$ of $G-v$ such that $d_{G[V(C)\cup \{v\}]}(v)$ is even.}
Let $G'=G-C$ and note that $d_{G'}(v)$ is even. Arguing by contradiction,  suppose $G'$ has 
$\min\text{ovl}(G')\leq 2$. Let $\varphi$ be an odd $3$-edge-covering of $G'$ with $\text{ovl}(\varphi)\leq 2$, and say the color $3$ appears in $E_{G'}(v)$ under $\varphi$. 

Note that $v$ is an internal vertex for the induced subgraph $G[V(C)\cup \{v\}]$. 
By Lemma~\ref{coforest}(a), 
there is a $V(C)$-join $H$ which is a coforest and contains the whole of
$E_{G[V(C)\cup \{v\}]}(v)$. In other words, letting $\hat H$ be the edge-complement of 
$H$ in $G[V(C)\cup \{v\}]$, it holds that $d_{\hat H}(v)=0$. Assign the color $3$ to all of $E(H)$ 
and take an odd $2$-edge-coloring of $\hat H$ with color set $\{1,2\}$.
Clearly, this procedure extends 
$\varphi$ to an odd $3$-edge-covering of $G$, with overlapping at most $2$, which is a contradiction.
Therefore, we conclude that $\min\text{ovl}(G')\geq 3$.

\smallskip

\noindent {\bf (b)} If $G-v$ contains no cut vertices, then Corollary~\ref{connectivity_at_least_3} implies that $\min\text{ovl}(G)\leq 2$, which is absurd.
So $G-v$ is not a block graph. Let $B$ be an end-block of  $G-v$, and let $s$ be the unique cut vertex of $G-v$ that is contained in $B$. Note that every vertex $w\in B-s$ is an internal vertex of $G-v$. 

\smallskip

\noindent \textbf{Claim 1.} \textit{$\mathrm{Int}_{G-v}B\cap N_G(v)\neq\emptyset$.} 

Suppose otherwise. Then $s$ is not only a cut vertex of $G-v$ but also a cut vertex of $G$. By Lemma~\ref{cover_odd_component}, $v(B)$ is odd. Let $G'=G-(B-s)$. Thus $v(G')$ is odd as well. Take a $(V(B)-s)$-join $H$ of $B$ which is a coforest and has $|E_B(s)\backslash E_H(s)|\leq1$; here we use that $s$ is an internal vertex of $B$. Similarly, take a $(V(G')-s)$-join $H'$ of $G'$ which is a coforest. Let $\hat{H}$ and $\hat{H'}$ be the corresponding edge-complements (with regard to $B$ and $G'$, respectively). Note that $d_{\hat{H}}(s)=0$ or $1$ depending on whether $d_B(s)$ is even or odd. We distinguish between the two possibilities. 

Assume first that $d_{\hat{H}}(s)=0$. Take an odd $3$-edge-covering $\varphi$ of $G'$ with $\text{ovl}(\varphi)= 0$ (if no such covering exists, we are done). Say that the color $1$ appears in $E_{G'}(s)$ under $\varphi$. Extend to $E(G)$ by coloring $E(H)$ with $1$ and using an odd $2$-edge-coloring of the forest $\hat{H}$ with the color set $\{2,3\}$. 

Now assume that $d_{\hat{H}}(s)=1$. Let $E_{\hat{H}}(s)=\{e\}$. Color $E(H')\cup\{e\}$ with $1$. Let $u$ be the other endvertex of $e$ (besides $s$). The precoloring of $E_{\hat{H}}(u)$ (the edge $e$ is already colored by $1$) extends to an odd $2$-edge-coloring of $\hat{H}$ with color set $\{1,2\}$. Color $E(H)$ by $3$. Let $F$ be a forest of $G'$ which is minimal subjected to the conditions that $\hat{H'}\subseteq F$ and $d_F(s)>0$. In other words, either take $F=\hat{H'}$ (if $d_{\hat{H'}}(s)>0$) or $F=\hat{H'}+f$ for some edge $f\in E_{G'}(s)=E_{H'}(s)$. Take an odd $2$-edge-coloring of $F$ with color set $\{2,3\}$ under which the color $3$ appears on $E_{F}(s)$. This completes an odd $3$-edge-covering $\psi$ of $G$ with $\text{ovl}(\psi)\leq2$. 
\hfill$\diamond$

\smallskip

In view of Claim~1, Lemma~\ref{claim23} implies the following. 
\begin{enumerate}
\item$v$ is adjacent to  every vertex of $B-s$. 
\item every vertex $w$ in $B-s$ satisfies that $d_{G-v}(w)$ is even. 
\end{enumerate}

\smallskip

\noindent \textbf{Claim 2.}  \textit{The vertex $s$ is adjacent to every vertex of $B-s$.} 

Suppose otherwise and let $w\in \mathrm{Int}_{G-v}(B)$ be such that $s$ is not adjacent to $w$.
Since $w$ has an odd degree, the neighbors of $w$ in $B-s$ can be listed
as $\{q_1,\cdots,q_{2\ell}\}$. Recall that $\mathcal{E}$ (respectively, $\mathcal{O}$) denotes the set of all even (respectively, odd) vertices of $G$. 
By comparing parities, we see that $|\mathcal O|$ is even, implying that $|\mathcal E|$ is odd. 
Consider the even-sized vertex subset 
$T=\mathcal{E}-\{v\}$ in the graph $G-(\{v\}\cup(B-s))$. By Lemma~\ref{forest_coforest}, there is a $T$-join $H$, which is also 
a forest. Take an odd $2$-edge-coloring of $H$ with color set $\{1,2\}$.
Assign the color $1$ to $wq_j$ and the color $2$ to $vq_j$ for each $j=1,\ldots, 2k$.
For the special edge $vw$, assign both the colors 1 and 2.  
The remaining uncolored edges of $G$ form an odd graph; assign the color $3$ to all of them. 
This defines an odd $3$-edge-covering $\varphi$ of $G$ with 
$\text{ovl}(\varphi)=2$, a contradiction. \hfill$\diamond$

\smallskip

Now, let us suppose, for contradiction, that the graph $G'=G-(B-s)$ admits an odd $3$-edge-covering $\psi$ with $\text{ovl}(\psi)\leq 2$. 
Since the induced subgraph $B-s$ is an odd graph, it must have even order. Let $V(B-s) = \{v_1,\cdots,v_{2k}\}$. Note that $N_G(v)- (B-s)\neq \emptyset$ (by Lemma~\ref{cover_odd_component})
In the covering $\psi$, if a color appears in both edge sets $E_{G'}(v)$ and $E_{G'}(s)$, let us say color $1$, then we can extend the covering to an odd $3$-edge-covering $\psi'$ of $G$ with overlapping at most 2 by assigning color $1$ to all the uncolored edges.
Otherwise, without loss of generality, we can assume that the edge set $E_{G'}(s)$ contains color $1$ and does not contain color $3$, while  $E_{G'}(v)$ contains color $3$ and does not contain color $1$.
Then, we assign color $3$ to all
edges $vv_j$ and assign color $1$ to all edges $sv_j$, for $j=1,2,\cdots,2k$, 
and then assign color $2$ to 
all the edges in $B-s$.
One can verify that this procedure extends 
$\psi$ to an odd $3$-edge-covering $\psi'$ of $G$ with overlapping at most $2$. 
In either case, we have obtained a contradiction. 
\end{proof}

\begin{proof}[Proof of Theorem~\ref{covering_conjecture}] With all the preparatory work completed as above, we proceed to prove the theorem by contradiction.
 To do this, we select a simple connected graph $G$ as our example, satisfying $\min \text{ovl}(G)\geq 3$ and having the minimum number of vertices. 
If $G$ contains at most $3$ vertices, then 
it contains at most $3$ edges, and clearly $\min \text{ovl}(G)=0$, which is absurd. 
Furthermore, since 
$v(G)=n$ is odd (as otherwise $\min \text{ovl}(G)=0$ by Theorem~\ref{Pyber_4}), we have $v(G)\geq 5$. By comparing parities, there exists at least one even vertex $v$.  
If $v$ is a cut vertex, we can apply Proposition~\ref{main_step} (a) to obtain an example $G'$ with $\min \text{ovl}(G')\geq 3$ with $|G'|\leq |G|-2$, which contradicts our choice of $G$. 
If $G-v$ is connected, then we apply Proposition~\ref{main_step} (b) to obtain an example $G'$ with $\min \text{ovl}(G')\geq 3$ with $|G'|\leq |G|-2$, which again contradicts our choice of $G$.
\end{proof}

\section{Concluding remarks}
For a pair $(G,e)$ consisting of a loopless graph $G$ and an edge $e\in E(G)$, let $G+e'$ denote the spannning supergraph of $G$, obtained by doubling the edge $e$; in other words, we enlarge $G$ with a new edge $e'$ which is parallel to $e$. For example, by doubling any edge of $K_3$ it becomes the reduced Shannon triangle $S_{2,1,1}$. An equivalent formulation of Theorem~\ref{main} is the following. 

\smallskip

\textit{Every loopless connected graph that is not $S_{2,2,2}$ or $S_{2,2,1}$ is either odd $3$-edge-colorable or there exists an edge $e\in E(G)$ such that the graph $G+e'$ admits an odd $3$-edge-coloring in which $e$ and $e'$ receive the same color.}

\smallskip

Similarly, an equivalent formulation of Theorem~\ref{covering_conjecture} is the following.

\smallskip

\textit{Every simple connected graph is either odd $3$-edge-colorable, or there exists an edge $e\in E(G)$ such that the graph $G+e'$ admits an odd $3$-edge-coloring in which $e$ and $e'$ receive different colors.}

\smallskip

Consequently, every simple connected graph $G$ with $\chi_{\text{odd}}(G)=4$ contains edges $e$ and $f$ (not necessarily distinct) such that:
\begin{enumerate}
\item[$(1)$] the graph $G+e'$ admits an odd $3$-edge-coloring under which $e$ and $e'$ are colored the same;
\item[$(2)$] the graph $G+f'$ admits an odd $3$-edge-coloring under which $f$ and $f'$ are colored differently.
\end{enumerate}

Figure~\ref{fig:wheel} depicts the wheel $W_4$, a simple connected graph with odd chromatic index $4$. Any of its four spokes satisfies both $(1)$ and $(2)$. One wonders whether every simple connected graph $G$ with odd chromatic index $4$ contains an edge $e=f$ for which both $(1)$ and $(2)$ hold. We are tempted to end this article with the following.

\smallskip

\noindent\textbf{Problem.} \textit{Find a simple connected graph $G$ with $\chi_{\text{odd}}'(G)=4$ such that for no edge $e\in E(G)$ the graph $G+e'$ admits two odd $3$-edge-colorings, $\varphi_1$ and $\varphi_2$, such that $\varphi_1(e)=\varphi_1(e')$ and $\varphi_2(e)\neq\varphi_2(e')$.}

\begin{figure}[ht!]
	$$
		\includegraphics[scale=1.21]{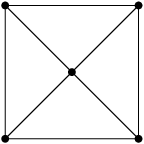}
	$$
	\caption{Any edge $e$ of the wheel $W_4$ satisfies $(1)$. Only a spoke edge $f$ satisfies $(2)$.}
	\label{fig:wheel}
\end{figure}

\bibliographystyle{plain}
\addcontentsline{toc}{chapter}{Bibliography}
\bibliography{odd_multi}
\end{document}